\documentclass[11pt,reqno,a4paper]{amsart}
\usepackage{mathrsfs}

\usepackage{amssymb,amsmath,amsfonts,amsthm,mathtools,mathabx}
\usepackage{placeins}   
\usepackage{mathtools}
\usepackage{calrsfs}
\usepackage[margin=1.0in]{geometry}
\usepackage{multirow}
\usepackage{enumerate,mdwlist}
\usepackage{hyperref}
\hypersetup{colorlinks=true}

\newtheorem{theorem}{Theorem}[section]
\newtheorem{lemma}[theorem]{Lemma}
\newtheorem{propos}[theorem]{Proposition}
\newtheorem{corol}[theorem]{Corollary}
\theoremstyle{remark}

\newtheorem{defin}[theorem]{Definition}
\newtheorem{rmk}[theorem]{Remark}

\numberwithin{equation}{section}

\allowdisplaybreaks

\newcommand{\dive}{{\rm{div}\,}}
\newcommand{\loc}{ \rm{loc}}

\newcommand{\La}{\mathcal{L}}
\newcommand{\var}{\varepsilon} 
\newcommand{\varp}{\varphi} 

\newcommand{\w}{\omega}
\newcommand{\curl}{{\rm curl  } \, }

\newcommand{\p}{\partial} 
\newcommand{\IR}{\mathbb{R}}
\newcommand{\irn}{\IR^N}
\newcommand{\ID}{\mathcal{D}}
\newcommand{\IS}{\mathcal{S}}

\begin{document}

\bibliographystyle{plain}

\title[Lagrangian solutions to the Euler system with $L^1$ vorticity]
{Lagrangian solutions to the 2D Euler system \\ with $L^1$ vorticity and infinite energy}

\author[A. Bohun]{Anna Bohun}
\address{Anna Bohun, Departement Mathematik und Informatik, Universit\"at Basel, 
Spiegelgasse 1, CH-4051, Basel, Switzerland}
\email{anna.bohun@unibas.ch}
\author[F. Bouchut]{Fran\c{c}ois Bouchut}
\address{Fran\c{c}ois Bouchut, Universit\'e Paris-Est, Laboratoire d'Analyse et de Math\'ematiques Appliqu\'ees (UMR 8050),
CNRS, UPEM, UPEC, F-77454, Marne-la-Vall\'ee, France}
\email{francois.bouchut@u-pem.fr} 
\author[G. Crippa]{Gianluca Crippa}
\address{Gianluca Crippa, Departement Mathematik und Informatik, Universit\"at Basel, 
Spiegelgasse 1, CH-4051, Basel, Switzerland}
\email{gianluca.crippa@unibas.ch}

\begin{abstract}
We consider solutions to the two-dimensional incompressible Euler system with only integrable
vorticity, thus with possibly locally infinite energy. With such regularity, we use the recently developed theory of Lagrangian flows associated
to vector fields with gradient given by a singular integral in order to define Lagrangian solutions,
for which the vorticity is transported by the flow.
We prove strong stability of these solutions via strong convergence of the flow, under the only assumption
of $L^1$ weak convergence of the initial vorticity. The existence of Lagrangian solutions
to the Euler system follows for arbitrary $L^1$ vorticity.
Relations with previously known notions of solutions are established.
\end{abstract}

\maketitle


\section{Introduction}

The incompressible Euler equations for a two-dimensional inviscid fluid are given by
\begin{align}\label{ve1}
\begin{cases}
& \p_t v + \dive(v\otimes v)  +\nabla p =0, \\
& v(0,\cdot)=v^0(x), \\
& \dive v =0 ,
\end{cases}
 \end{align}
where $v(t,x)$ is the velocity of particles at position $x$ and time $t$, and $p(t,x)$ the scalar pressure, that sustains the incompressibility constraint $\dive v = 0$.
The two-dimensional incompressible Euler equations may be rewritten as a transport equation for the scalar vorticity $\w$ defined by
\begin{equation}\label{v}
\w= \curl  v \equiv \partial_1 v_2-\partial_2 v_1,
\end{equation}
which is advected by the velocity $v$. 
This gives the vorticity formulation
\begin{equation}\label{vorticityformsmooth}
\begin{aligned}
\begin{cases}
&\p_t\w+ {\dive}(\w v)=0, \\
&\w(0,\cdot)=\w^0(x),
\end{cases}
\end{aligned}
\end{equation}
with $\w^0 = \curl v^0$.
The coupling \eqref{v} can alternatively be written via the Biot-Savart convolution law
\begin{equation}\label{biotsavart}
v(t,x) =\frac{1}{2\pi}\int\limits_{\IR^2}\frac{(x-y)^\perp}{|x-y|^2}\w(t,y) \, dy = K\mathop{*}_{x}  \w,
\end{equation}
where we denote by $(x_1,x_2)^\perp=(-x_2,x_1)$ and by
\begin{equation}\label{biotsavartkernelK} 
K(x)=\frac{1}{2\pi}\frac{x^\perp}{|x|^2}=\frac{1}{2\pi}\left(\frac{-x_2}{|x|^2},\frac{x_1}{|x|^2}\right)
\end{equation}
the Biot-Savart kernel.\\

In this paper we deal with the existence and stability of infinite kinetic energy solutions associated to initial vorticities lying in $ L^1(\IR^2)$.
In this context, because of the lack of (even local) kinetic energy bound, the velocity formulation \eqref{ve1} cannot be given the usual distributional meaning (see Definition~\ref{velocityform}).
Though, a {\em symmetrized velocity formulation} can be used, see Definition~\ref{symmetrization}.

For vorticities $\omega \in L^\infty((0,T);L^1(\IR^2))$, the decomposition
\begin{equation}\label{vdecomp}
v  =  K_1 * \w + K_2 * \w ,
\end{equation}
where $K_1 = K \, \mathbf{1}_{B_1(0)} \in L^1(\IR^2)$ and $K_2 = K\,\mathbf{1}_{B_1(0)^c} \in L^\infty(\IR^2)$,  gives immediately with Young's inequality
that $v\in L^\infty((0,T);L^1(\IR^2))+L^\infty((0,T);L^\infty(\IR^2))$.
Nevertheless, as for the velocity formulation, no direct distributional formulation is available for the vorticity equation \eqref{vorticityformsmooth},
since the factors in the product $\w v$ are not summable enough to define a locally integrable product.
The symmetrized formulation can however be used again.
More generally, one can consider three alternate formulations of weak solutions for the vorticity equation, defined as follows.
\begin{enumerate}
\item \emph{Renormalized} solutions \cite{renorm}, defined by the requirement that $\beta(\w)$ is a distributional solution to the transport equation \eqref{vorticityformsmooth} for a suitable class of functions $\beta$:
\begin{equation}
\begin{cases}
\begin{aligned} 
& \p_t(\beta(\w))+\dive (\beta(\w)v)=0,\\
& \beta(\w)(0,\cdot)=\beta(\w^0),
\end{aligned}
\end{cases}
\end{equation}
\item \emph{Symmetrized vorticity} solutions \cite{delort,vecchiwu,schochet}, defined by exploiting the  antisymmetry of the Biot-Savart kernel $K$, so that multiplying \eqref{vorticityformsmooth} by a test function $\phi$ and integrating gives the formulation
\begin{equation} \label{symintro}
\int^T_0 \int\limits_{\IR^2} \partial_t\phi(t,x)\w(t,x)\,dxdt-\int^T_0 \int\limits_{\IR^2} \int\limits_{\IR^2}  {H}_\phi(t,x,y)\w(t,x)\w(t,y)\,dxdydt +\int\limits_{\IR^2} \phi(0,x)\w^0(x)\,dx=0,
\end{equation}
where ${H}_\phi$ is the bounded function
\begin{equation}
	{H}_\phi(t,x,y)=-\frac{1}{2}K(x-y)\cdot\bigl(\nabla\phi(t,x)-\nabla\phi(t,y)\bigr),
	\label{eq:h_phi}
\end{equation}
\item \emph{Lagrangian} solutions, i.e.~solutions $\w$ transported by a suitable flow associated to the velocity $v$, to be precisely defined in the sequel.
\end{enumerate}
The notions (1) and (2) of solutions have been considered previously, but no study has
been made concerning (3) in our low regularity context $\w\in L^1$.
In this paper we prove that initial vorticities in $L^1$ give rise to well-defined weak solutions that are transported by flows.  
The key point of our strategy relies on a priori error estimates (under bounds that are natural in our setting) 
for a class of flows which are measure preserving, called \emph{regular Lagrangian flows}.
These estimates were developed in \cite{jhde}.
The novelty of this approach for the Euler equations, in contrast with \cite{delort,vecchiwu,indiana,evansmueller},  is that it entirely relies on the Lagrangian formulation, and therefore proves existence of solutions which are naturally associated to flows. In this setting we also allow for velocities with locally infinite kinetic energy.

The usual strategy for proving existence of solutions to \eqref{ve1} is by smoothing the initial data, and using estimates that enable passing to the limit in the weak formulation.
For initial velocities belonging to $H^s, s>2$, well-posedness of classical solutions is due to Wolibner \cite{wolibner}. Existence and uniqueness of solutions to \eqref{ve1} is known for vorticities in $L^1\cap L^\infty$, and was first proved by Yudovich \cite{yudovich}. For compactly supported initial vorticities in $L^p$, with $1<p<\infty$, existence was first proved by DiPerna and Majda \cite{dipernamajda1}. The proof relies on suitable Sobolev embeddings, that guarantee strong convergence in $L^2_{\loc}$ on the approximate velocities.

On the other hand, while a uniform $L^1$ bound on the vorticities is still sufficient to guarantee the $L^1_{\loc}$ convergence of the smoothed velocities, it is generally insufficient for the strong convergence in $L^2_{\loc}$, see for instance Example~11.2.1 in \cite{Majda}: the approximate velocities may concentrate.
However, concentrations may occur for sequences  whose limit \emph{still} satisfies \eqref{velocityform}, in spite of the lack of strong $L^2_{\loc}$ convergence: this is referred to as concentration-cancellation and has been studied in \cite{dipernamajdaconcel,Majda}. 
This happens for instance if the vorticity is a measure with distinguished sign \cite{delort}.
The key point in proving that concentration-cancellations occur is to prove distributional convergence of the antisymmetric quantities $v^1_n v^2_n$ and $v^1_n-v^2_n$.
For $L^1$ vorticities with compact support, without  necessarily distinguished sign, and initial velocities with locally finite kinetic energy, the propagation of the equi-integrability guarantees concentration-cancellations \cite{vecchiwu}.
However, these solutions were not proved to be Lagrangian, and only weak $L^1$ convergence was obtained on $\w$ even for strongly convergent initial data.
This is nevertheless sufficient to pass to the limit in the symmetrized formulation \eqref{symintro}.\\

A stability estimate for flows associated to velocity fields with gradient given by the singular integral of an $L^1$ function was derived in \cite{jhde}, building on previous results in \cite{cdel}.
This regularity of the field is weaker than the one classically used, namely $W^{1,1}$ or BV \cite{renorm,ambrosio}.
Our assumptions in the context of Euler equations fall under the theory in \cite{jhde}. From this theory it follows that Lagrangian flows associated to velocities whose curl are equi-integrable are strongly precompact, and thus stable under approximation, so that the limit flow solves the ODE with the limit velocity. 
We shall therefore conclude that vorticities in $L^1$ are strongly stable under approximation, in the sense that if $\w^0_n$ converges strongly in $L^1$ to  $\w^0$, then the solution $\w_n$ of the corresponding vorticity formulation converges strongly in $L^1$ to a \emph{Lagrangian} solution $\w$.
Additionally, even for weakly convergent initial vorticities, the flow always converges strongly.
The main results of this paper were announced in \cite{BCX}.

A classical difficulty in proving strong compactness is related to time oscillations.
Indeed, when dealing with velocity formulations, the strong compactness in space follows from
the $L^1$ bound on the vorticity, but the compactness in time relies on bounds on $\p_t v_n$ in $L^\infty_t({\mathcal D}'_x)$ in order for Aubin-Lions' lemma to apply. Without the assumption $v\in L^2_{\loc}$, we do not have such  regularity in time of $v$ and we cannot apply Aubin-Lions' lemma.
We thus propose a refinement of the stability estimates in \cite{jhde} so that weak time convergence of the velocities is still sufficient for the stability of regular Lagrangian flows.
We nevertheless prove a posteriori the strong compactness of $v$ in time and space.

\subsection*{Main notations}
We set $B_R:=B_R(0)$. We denote by $L^0(\IR^d)$ the space of all measurable real valued functions on $\IR^d$, defined a.e.~with respect to the Lebesgue measure, endowed with the convergence in measure. We denote by $L^0_{\loc}(\IR^d)$ the same space, endowed with local convergence in measure (see definition below). The space $\log L(\IR^d)$ contains all functions $u:\IR^d\to \IR$ such that $\int_{\IR^d} \log(1+|u(x)|) \, dx< \infty$, with $\log L_{\loc}(\IR^d)$ defined accordingly. We refer to $\mathcal{B}(E,F)$ as the space of bounded functions between sets $E$ and $F$.
We also introduce the following seminorm:
\begin{defin}
Let $u$ be a measurable function on $\Omega\subset \IR^d$. For $1\leq p<\infty$, we set
$$
|||u|||^p_{M^p(\Omega)}=\sup_{\lambda>0} \Big\{\lambda^p \La^d \big(\{ x \in \Omega \; : \; |u(x)|>\lambda\} \big) \Big\}
$$
and define the weak Lebesgue space $M^p(\Omega)$ as the space consisting of all such measurable functions $u:\Omega\rightarrow \IR$ with $|||u|||_{M^p(\Omega)}<\infty$. For $p=\infty$, we set $M^\infty(\Omega)=L^\infty(\Omega)$.
\end{defin}

\begin{defin} We say that a sequence of measurable functions $u_n: \IR^d\to\IR$ converges locally in measure in $\IR^d$ to a measurable function $u:\IR^d\to\IR$ if for every $\gamma>0$ and every $r>0$ there holds
$$\La^N(\{x\in B_r: |u_n(x)-u(x)|>\gamma\})\to 0, \qquad n\to \infty.$$
\end{defin}

\section{Regularity of the velocity field}\label{regv}
We summarize in the present section some integrability and regularity estimates for the vector field $v$ given by \eqref{biotsavart}, when $\w\in L^1(\IR^2)$. 

\begin{enumerate}[(I)]
\item The Biot Savart kernel $K$ belongs to $ L^1_{\loc}(\IR^2)$ and has distributional derivatives given by the following singular kernels. For $i,j=1,2$, we have
\begin{equation}\label{dK}
\p_{j} K^i(x) = \p_{j}\frac{1}{2\pi} \left(\frac{-x_2}{|x|^2},\frac{x_1}{|x|^2}\right)_i.
\end{equation}
The Fourier transform of \eqref{dK} is bounded and is given by
\begin{equation}
\widehat{ \p_{j} K^i}(\xi) =  \xi_{j} \left(\frac{-\xi_2}{|\xi|^2},\frac{\xi_1}{|\xi|^2}\right)_i \in L^\infty(\IR^2).
\end{equation}
It is well-known that the operators $S^i_j$ of convolution with such kernels \eqref{dK}
defined by $S^i_ju=\p_{j} K^i*u$ extend to bounded operators on $L^2(\IR^2)$, and have bounded extensions on $L^p(\IR^2)$ for $1<p<\infty$.
For $p=1$ and $u\in L^1(\IR^2)$, $S^i_ju$ is a tempered distribution $S^i_j u\in \IS'(\IR^2)$ defined via the formula
\begin{equation} 
\langle S^i_j u,\varphi\rangle=\langle u,\tilde{S}^i_j\varp\rangle \qquad
\forall \varp\in\IS(\IR^2),
\end{equation}
where $\tilde{S}^i_j$ is the singular integral operator associated to the kernel $(\p_{j} K^i)(-x)$.
Thus for $v$ given by \eqref{biotsavart} with $\w\in L^1(\IR^2)$ and for $i,j=1,2$, we have 
\begin{equation}\label{Sw}
(\nabla v)_{ij}= \p_{j} v^i=S_j ^i\w\in \IS'(\IR^2).
\end{equation}
\item
{}From \eqref{Sw} it follows that
 $\dive v=0$ in $\ID'$, and $\curl v=\w$ in $\ID'$.
\item 
We have from \eqref{vdecomp} that vorticities bounded in $L^\infty((0,T);L^1(\IR^2))$ are associated to velocities bounded in $   L^\infty((0,T);L^1(\IR^2))+ L^\infty((0,T);L^\infty(\IR^2)).$
Moreover, the weak Hardy-Littlewood-Sobolev inequality (see Lemma 4.5.7 in \cite{hormander}) gives that
\begin{equation}\label{M2}
 \|v\|_{L^\infty((0,T);M^2( \IR^2))} \leq  c \|\w \|_{L^\infty((0,T); L^1(\IR^2))} ,
\end{equation}
which implies in particular the embedding
$v(t,x)\in   L^p_{\loc}([0,T]\times \IR^2)$ for any $1\leq p<2$.
\end{enumerate}

\section{Weak solutions}\label{weaksoleulerl1}
Several weak formulations can be considered. If the velocity has locally finite kinetic energy, $v \in L^2_{\loc}(\IR^2)$, the usual weak formulation of \eqref{ve1} is available: 

\begin{defin}[Weak velocity formulation]\label{velocityform} 
We say that $v\in L^\infty((0,T);L^2_{\loc}(\IR^2) )$ is a weak solution of the Euler velocity formulation with initial datum $v^0\in L^2_{\loc}(\IR^2)$ if for all $\phi(t,x)\in C_c^1([0,T)\times\IR^2,\IR^2)$ with $\dive\phi =0$, there holds
\begin{equation}\label{distributionalform}
 \int_0^T \int\limits_{\IR^2} \partial_t\phi \cdot v + \nabla\phi :\bigl(v\otimes v\bigr) \, dxdt+ \int\limits_{\IR^2} \phi(0,x)\cdot v ^0(x) \, dx =0,
\end{equation}
and $v$ is divergence free in distributional sense.
\end{defin}

\subsection{Symmetrized velocity solutions.}
In order to deal with solutions with locally infinite kinetic energy we can propose a weaker formulation than the one in Definition~\ref{velocityform}.
It is in the same spirit as the symmetrized vorticity formulation \eqref{symintro}.
Using the identity $\dive(v \otimes v ) =v\cdot\nabla v= \w \, v^\perp  + \nabla\frac{ |v|^2}{2}$,
that is valid when $\dive v=0$, we can formally rewrite \eqref{ve1} as
\begin{equation} 
\p_t v+\w  v^\perp +\nabla p' =0,
\label{eq:vmodified}
\end{equation}
where $p'=p+\frac{|v|^2}{2}$. This modified pressure $p'$ can be eliminated by
taking suitable test functions as in \eqref{distributionalform}.
With this form \eqref{eq:vmodified} we can observe that only the quantities $v_1v_2$
and $v_1^2-v_2^2$ need to be in $L^1$, since we can write
$\w  v^\perp=\dive(v\otimes v-(|v|^2/2){\mathrm Id})$,
and the entries of the matrix $v\otimes v-(|v|^2/2){\mathrm Id}$ are just these two scalars
$v_1v_2$ and $v_1^2-v_2^2$.
However, without such assumptions, we observe that the term $\w v^\perp$ has a priori no pointwise meaning
when $\w$ only belongs to $L^p$ for some $p<4/3$, since in such a case $\omega$ and $v$ would not have conjugate summabilities.
Nevertheless, with the only assumption $\w\in L^1$, that yields $v\in M^2$ (but $v\not\in L^2_{\loc}$ in general),
we can give a meaning in distribution sense to this term by exploiting the symmetrization technique analog to that in \cite{delort,vecchiwu},
that uses the antisymmetry property $K(-x)=-K(x)$.

Let $\phi\in C_c^1([0,T)\times\IR^2,\IR^2)$. Then using the Biot-Savart law we can write
\begin{equation}\label{symm}
 \begin{aligned} 
& \int^T_0 \int\limits_{\IR^2} (\w v^\perp)(t,x) \cdot\phi(t,x) \, dxdt  \\
 & = \int^T_0\int\limits_{\IR^2}\int\limits_{\IR^2}\w(t,x) \w(t,y)K(x-y)^\perp\cdot\phi(t,x) \, dx dy dt\\
 & =   -\int^T_0\int\limits_{\IR^2}\int\limits_{\IR^2} \w(t,y)  \w(t,x)K(x-y)^\perp\cdot\phi(t,y)  \,dx dy dt  \\
 & = \frac{1}{2}\int^T_0\int\limits_{\IR^2} \int\limits_{\IR^2}  \w(t,x)\w(t,y)K(x-y)^\perp  \cdot\bigl(\phi(t,x)- \phi(t,y)\bigr) \, dx dy dt\\
 &=\int^T_0\int\limits_{\IR^2} \int\limits_{\IR^2}  \w(t,x)\w(t,y) \bar{H}_{\phi}(t,x,y) \,dx dydt,
 \end{aligned}
 \end{equation}
where $\bar{H}_\phi(t,x,y)$ is the function on $[0,T)\times\IR^2\times\IR^2$ given by  
\begin{equation}\label{barh}
\bar{H}_\phi(t,x,y)= \frac{1}{2} K(x-y)^\perp \cdot\bigl( \phi(t,x)- \phi(t,y)\bigr).
\end{equation} 
For $\phi\in C^1_c([0,T)\times\IR^2,\IR^2)$ we have that $\bar{H}_\phi$ is a bounded function, continuous outside the diagonal, that tends to zero at infinity.  Indeed we have
\begin{equation}
	|\bar{H}_\phi(t,x,y)| \leq \frac{1}{4\pi }  \textrm{Lip}( \phi(t,\cdot)).
	\label{eq:estHphi}
\end{equation}
Thus for vorticities belonging to $L^\infty((0,T);L^1(\IR^2))$, the last integral in \eqref{symm} is well-defined. 
This motivates the next definition of weak solutions.
\begin{defin}[Symmetrized velocity formulation]\label{symmetrization}
Let $(\w^0,v^0)\in  L^1(\IR^2)\times M^2(\IR^2)$, with $\w^0=\curl v^0$. 
We say that the couple $(\w,v)$ is a symmetrized velocity solution of \eqref{ve1} 
in $[0,T)$ with initial datum $(\w^0,v^0)$, if
\begin{enumerate}
\item $\w\in L^\infty((0,T);L^1(\IR^2))$,
\item the velocity field $v$ is given by the convolution in \eqref{biotsavart},
\item for all test functions $\phi\in C^1_c([0,T)\times \IR^2,\IR^2)$ with $\dive\phi=0$, we have
\begin{equation}\label{weakformh}
\int^T_0 \int\limits_{\IR^2} \partial_t\phi\cdot v \, dxdt
-\int^T_0 \int\limits_{\IR^2} \int\limits_{\IR^2} \bar{H}_\phi(t,x,y)\w(t,x)\w(t,y) \,dxdydt 
+\int\limits_{\IR^2} \phi(0,x)\cdot v^0(x) \, dx=0,
\end{equation}
where $\bar{H}_\phi$ is the function on $[0,T)\times\IR^2\times\IR^2$ given by  \eqref{barh}.
\end{enumerate}
\end{defin}

\subsection{Three formulations of the vorticity equation}\label{vorticityformulations}

According to the introduction, we now define three notions of solution to the vorticity formulation \eqref{vorticityformsmooth} when the vorticity is only $L^1$ summable. 
Since we do not assume $v^0\in L^2_{\loc}(\IR^2)$, we deal with velocities that belong to $M^2(\IR^2)$, a consequence of the Hardy-Littlewood inequality \eqref{M2}.

\begin{defin}[Renormalized solutions]\label{renorm} Let $(\w^0,v^0)\in  L^1(\IR^2)\times M^2(\IR^2)$ with $\w^0= \curl v^0$.  We say the couple $(\w,v)$ is a renormalized solution to \eqref{vorticityformsmooth} with initial data $(\w^0,v^0)$, if
\begin{enumerate}
\item $\w\in L^\infty((0,T);L^1(\IR^2))$,
\item the velocity field $v$ is given by the convolution in \eqref{biotsavart},
\item for every nonlinearity $\beta\in C^1(\IR)$ with $\beta$ bounded, we have that
\begin{equation}
\begin{cases}
\begin{aligned} 
& \p_t(\beta(\w))+\dive (\beta(\w)v)=0,\\\
& \beta(\w)(0,\cdot)=\beta(\w^0) 
\end{aligned}
\end{cases}
\end{equation}
hold in the sense of distributions. 
\end{enumerate}
\end{defin}
For smooth solutions this is equivalent to the classical notion of solution (as can be seen by multiplying the equation by $\beta'(\w)$ and applying the chain rule.) This formulation derives from the classical DiPerna-Lions \cite{renorm} framework for transport equations.

\begin{defin}[Symmetrized vorticity formulation]\label{symvort} As mentioned in the introduction, the symmetrization technique for the term $\dive(\w v)$ provides a second formulation of the vorticity equation.
Let $\phi\in C_c^2([0,T)\times\IR^2)$. Computations as in \eqref{symm} give  
\begin{equation}\label{sym}
 \begin{aligned} 
& \int^T_0 \int\limits_{\IR^2} \dive(\w v)(t,x) \phi(t,x) \,  dxdt =\int^T_0\int\limits_{\IR^2} \int\limits_{\IR^2}  H_\phi(t,x,y)\w(t,x)\w(t,y)\, dx dydt,
\end{aligned}
\end{equation}
with
\begin{equation}\label{h}
H_{\phi}(t,x,y)= -\frac{1}{2} K(x-y)\cdot\bigl(\nabla\phi(t,x)-\nabla\phi(t,y)\bigr).
\end{equation} 
We say that $(\w,v)$ is a symmetrized vorticity solution to \eqref{vorticityformsmooth} if (1), (2) above are satisfied
and if for all test functions $\phi\in C^2_c([0,T)\times \IR^2)$ there holds
\begin{equation}\label{weakform2}
\int^T_0 \int\limits_{\IR^2} \partial_t\phi(t,x)\w(t,x) \, dxdt-\int^T_0 \int\limits_{\IR^2} \int\limits_{\IR^2} {H}_\phi(t,x,y)\w(t,x)\w(t,y) \,dxdydt +\int\limits_{\IR^2} \phi(0,x)\w^0(x) \, dx=0.
\end{equation}
\end{defin}

\begin{propos}\label{prop equiv-sym-weak}
We have the following equivalence of notions of solutions to the Euler system.
\begin{enumerate}
\item Symmetrized velocity solutions (Definition~\ref{symmetrization})
are symmetrized vorticity solutions (Definition~\ref{symvort}), and conversely.
\item If $(\w,v)$ is such that $v\in L^\infty((0,T);L^2_{\loc}(\IR^2))$,
then it is a symmetrized velocity solution if and only if it is a weak velocity solution (Definition \ref{velocityform}).
\end{enumerate}
\end{propos}
\begin{proof}
For (1), taking a test function of the form $-\nabla^\perp \phi$ in \eqref{weakformh}
we see that a solution to the symmetrized velocity formulation
is also a solution to the symmetrized vorticity formulation,
indeed one has $\bar{H}_{-\nabla^\perp\phi}={H}_\phi$.
The converse is also true since all functions $\bar\phi\in C^2_c([0,T)\times \IR^2,\IR^2)$ with $\dive\bar\phi=0$ can be written
$\bar\phi=-\nabla^\perp\phi$ for some $\phi\in C^2_c([0,T)\times \IR^2)$. For $\bar\phi$ only $C^1$
one just approximates it by a $C^2$ function.
It follows that Definitions \ref{symmetrization} and \ref{symvort} are indeed equivalent.
Finally, the statement (2) follows from the next lemma.
\end{proof}

\begin{lemma}\label{lemma sym-weak}
Let $\w\in L^1(\IR^2)$, define $v=K*\w$ with $K$ the Biot-Savart kernel \eqref{biotsavartkernelK},
and assume that $v\in L^2_{\loc}(\IR^2)$. Then for all $\phi\in C^1_c(\IR^2,\IR^2)$ with $\dive\phi=0$, we have
\begin{equation}
\int\limits_{\IR^2} \int\limits_{\IR^2} \bar{H}_\phi(x,y)\w(x)\w(y) \,dxdy
=-\int\limits_{\IR^2}\nabla\phi(x):\bigl(v(x)\otimes v(x)\bigr) \, dx
\end{equation}
where $\bar{H}_\phi$ is given by \eqref{barh}.
\end{lemma}
\begin{proof} For smooth $\w$ and $v$, the formula is just the weak form
of the already mentioned identity $\dive(v \otimes v ) = \w \, v^\perp  + \nabla\frac{ |v|^2}{2}$,
taking into account the computation \eqref{symm}. The general case follows easily
by smoothing $\w$ and $v$ by a regularizing kernel and passing to the limit.
\end{proof}

\subsection{Lagrangian solutions}
We describe now a third class of weak solutions which are transported by a measure-preserving flow in an ``almost everywhere'' sense.
When the velocity is not globally bounded, the associated flow $X$ is not locally integrable in $\IR^2$,
thus the ODE defining the flow has to be taken in the renormalized sense.

Let us recall the following general definition on $\IR^N$.
Assume that a vector field $b(t,x):(0,T)\times\irn\rightarrow\irn$ can be decomposed as
\begin{equation}\label{r1}\tag{{\bf R1}}
\frac{b(s,x)}{1+|x|}=\tilde{b}_1(s,x)+\tilde{b}_2(s,x),
\end{equation}
with
\begin{equation}
\tilde{b}_1\in L^1((0,T);L^1(\irn)),\,\,\,\tilde{b}_2\in L^1((0,T);L^\infty(\irn)).
\end{equation}

\begin{defin}\label{deffloweuler}
If $b$ is a vector field satisfying {\bf (R1)}, then for fixed $t\in[0,T)$, a map $X(s,t,x)$ satisfying
\begin{equation}
	(s,x)\mapsto X(s, t, x) \in C([t,T]_s;L^0_{\loc}(\irn_x))\cap \mathcal{B}([t,T]_s;\log L_{\loc}(\irn_x))
	\label{eq:regX}
\end{equation}
is a regular Lagrangian flow (in the renormalized sense) relative to $b$ starting at $t$ if we have the following:

\begin{enumerate}[(i)]
\item The equation 
\begin{equation}\label{ode} \p_s\big(\beta(X(s,t,x))\big)=\beta ' (X(s,t,x)) b(s,X(s,t,x))
\end{equation}
holds in $\ID'((t,T)\times\irn)$, for every function $\beta\in C^1(\irn;\IR)$ that satisfies
$|\beta(z)|\leq C(1+\log(1+|z|))$ and $|\beta'(z)|\leq C/(1+|z|)$ for all $z\in\irn$, for some constant $C$,
\item  $X(t,t,x)=x$ for $\La^N$-a.e $x\in\irn$, 
\item There exists a compressibility constant $L\geq 0$ such that $\int_{\irn} \varp(X(s,t,x)) \, dx\leq L\int_{\irn} \varp(x) \, dx$ for all measurable $\varp:\irn\rightarrow[0,\infty)$.  
\end{enumerate}
\end{defin}
By now this is the usual definition of flows for weakly differentiable vector fields satisfying the general growth condition {\bf(R1)}.

We next consider the condition that the components of $\nabla b$ can be written as singular integrals of $L^1$ functions,
\begin{equation}\label{r2}\tag{{\bf R2}}
\partial_j b^i=S_j^i g_j^i\quad\mbox{in }{\mathcal D}'((0,T)\times\IR^N),
\end{equation}
where $S_j^i$ are singular integral operators in $\IR^N$, and $g_j^i\in L^1((0,T)\times\IR^N)$.

Finally we consider the conditions
\begin{equation}\label{R3}\tag{{\bf R3}}
b\in L^p_{\loc}([0,T]\times\IR^N)\quad\mbox{for some }p>1,
\end{equation}
and
\begin{equation}\label{R4}\tag{{\bf R4}}
\dive b\in L^1((0,T);L^\infty(\IR^N)).
\end{equation}
According to \cite{jhde}, under the assumptions {\bf(R1)}, {\bf(R2)}, {\bf(R3)}, {\bf(R4)},
a regular Lagrangian flow $X$ as in Definition \ref{deffloweuler}, except that now $s\in[0,T]$ instead of $s\in[t,T]$
(the forward-backward flow defined in Corollary 6.6 in \cite{jhde}) exist and is unique and stable.\\

With this notion of flow, we can define in accordance with \cite{jhde} our class of Lagrangian solutions $(\w,v)$ to the Euler equations by the relation
\begin{equation}\label{pushforwardomega}
	\w(t,x) = \w^0\Bigl(X(   s=0,t,x ) \Bigr), \qquad \mbox{for all } t\in [0,T].
\end{equation}

\begin{defin}[Lagrangian solution]\label{lagrangiansol}
Let $(\w^0,v^0)\in  L^1(\IR^2)\times M^2(\IR^2)$ with $\w^0=\curl v^0$.   We say the couple $(\w,v)$ is a Lagrangian solution to \eqref{vorticityformsmooth} in $[0,T]$ with initial data  $(\w^0,v^0)$, if
\begin{enumerate}
\item  $\w\in  C([0,T];L^1(\IR^2))$,
  \item the velocity field $v$ is given by the convolution in \eqref{biotsavart},
\item for all $t\in [0,T]$, $\w$ is given by the formula in \eqref{pushforwardomega}, where $X$ is the regular Lagrangian flow associated to $v$.
\end{enumerate}
 \end{defin}
Note that according to the properties stated in Section \ref{regv}, the vector field $b=v$
satisfies the properties {\bf(R1)}, {\bf(R2)}, {\bf(R3)}, {\bf(R4)}, with $g_j^i=\w$, justifying the existence of $X$.
We remark that according to \cite{jhde}, Lagrangian solutions in the sense of Definition \ref{lagrangiansol}
are also renormalized solutions in the sense of Definition \ref{renorm}.

\section{Strong stability of Lagrangian flows} 
We now recall from \cite{jhde} the following stability result.
It gives a quantitative estimate in measure on the flows difference,
in terms of the $L^1$ norm of the difference of the vector fields. 

\begin{theorem}[Fundamental estimate for flows]\label{fundest} Let $b$ and $\bar{b}$ be two vector fields,
$b$ satisfying assumptions {\bf (R1)-(R3)} and $\bar{b}$ satisfying only {\bf (R1)}.
Fix $t\in [0,T)$ and let $X$ and $\bar{X}$ be regular Lagrangian flows starting at time $t$
associated to $b$ and $\bar{b}$ respectively, with compressibility constants $L$ and $\bar{L}$.
Then for every $\gamma>0$ and $r>0$, and for every $\eta>0$ there exist $\lambda>0$ and $C_{\gamma,r,\eta}>0$ such that
\begin{equation}
\La^N(B_r\cap \{|X(s,\cdot)-\bar{X}(s,\cdot)|>\gamma\})\leq C_{\gamma,r,\eta}\|b-\bar{b}\|_{L^1((0,T)\times B_\lambda)}+\eta \qquad \textrm{for all } s\in [t,T].
\label{eq:mainest}
\end{equation}
The constant $\lambda$ and $C_{\gamma,r,\eta}$ depend on $\gamma,r,\eta$ and on the bounds on the operator norms of $S_j^i$, the norms involved in the estimates from {\bf (R1)} and {\bf (R3)}, the compressibility
constants $L$ and $\bar{L}$ of $X$ and $\bar{X}$, and the equi-integrability of $g_j^i$ from assumption {\bf (R2)} for $b$.
\end{theorem}
In previous literature (see again \cite{Majda}), strong $L^1_{\loc}$ convergence of smoothed velocities was guaranteed for initial data $v^0$ belonging to $L^2_{\loc}(\IR^2)$.
In order to allow for solutions with infinite kinetic energy, we bypass this assumption and use the weaker $M^2$ estimate arising in \eqref{M2}.
As seen in the estimate \eqref{eq:mainest}, we need the strong convergence in
time and space of the vector field. We have a priori only compactness in space, and we shall
therefore use a general argument to deal with only weak convergence in time.
We shall show a posteriori that given equi-integrable vorticity data, the associated velocities are indeed strongly compact in time and space.
An alternative way to get compactness is also explained in Remark \ref{rem time-compact}.\\

We now expand in the following Proposition~\ref{weakstability} a remark from \cite{cdel} proving that \emph{weak} convergence is sufficient for stability,
and adapt the proof from the setting of Sobolev regularity to the regularity given by {\bf(R2)}.
We begin with two lemmas, the first arising from standard analysis.

\begin{lemma} \label{tauh}
Let $K$ be the Biot-Savart kernel \eqref{biotsavartkernelK}, and denote by $\tau_h K(x)  = K(x+h)$.
Then for any $1<p<2$ and all $h\in\IR^2$ one has
\begin{equation}
\|\tau_h K-K\|_{L^p(\IR^2)} \leq c_p |h|^\alpha
\end{equation}
with $\alpha=2/p-1>0$. In particular, the linear mapping $L^1(\IR^2)\to L^1_{\loc}(\IR^2)$ defined by $g\mapsto K*g $ is a compact operator.
\end{lemma}
\begin{proof}
See Lemma 8.1 in \cite{vpe}  for the proof of the first inequality.
Next, denote $Tg=K*g$, and take an exponent$1<p<2$. Whenever $\|g\|_{L^1(\IR^2)}\leq 1$,
$Tg$ is bounded in $L^1_{\loc}$ and one has
\begin{equation}\label{tauvn}
\begin{aligned}
\left\| \tau_h (Tg)-Tg  \right\|_{L^p(\IR^2)}
& =   \left\|\left(\tau_h K-K\right)*g  \right\|_{L^p(\IR^2)}   \\
&\leq \left\|  \tau_h K-K \right\|_{L^p(\IR^2)}    \\
&\leq  c_p  |h|^\alpha . 
\end{aligned}
\end{equation}
Thus $\tau_h (Tg) - Tg$ is uniformly small in $L^1_{\loc}$ as $h\to 0$. Applying the Riesz-Fr\'echet-Kolmogorov criterion gives the result.
\end{proof}

The second lemma states that given classical flows associated to Lipschitz vector fields, weak convergence of the vector fields suffices for the associated flows to converge uniformly.

\begin{lemma}\label{remarkcdel} Let $b_n$ be a sequence of vector fields uniformly bounded in  $L^\infty((0,T)\times\IR^N)$
with $ \nabla_x b_n$ uniformly bounded in $L^\infty((0,T)\times \IR^N)$. 
Assume that there exists a vector field $b\in L^\infty((0,T)\times\IR^N)$ with $\nabla_x b\in L^\infty((0,T)\times \IR^N)$,
such that $b_n\rightharpoonup^* b$ in $L^\infty ((0,T)\times\IR^N)-{\mathrm w}*$.
Let $X_n(s,t,x)$ and $X(s,t,x)$ be Lagrangian flows (in the DiPerna-Lions sense)
associated to $b_n$ and $b$. Then $X_n(s,t,x)\to X(s,t,x)$ in $L^\infty((0,T)^2; L^\infty_{\loc}(\IR^N))$.
\end{lemma}

\begin{proof}
It follows from the uniform bounds on $b_n$ and $\nabla_x b_n$ that $\nabla_x X_n$ is uniformly bounded, with
\begin{equation}
	|\nabla_x X_n(s,t,x)|\leq \exp \big( T \|\nabla_x b_n\|_{L^\infty((0,T) \times \IR^N)} \big).
	\label{eq:estgradXn}
\end{equation}
Then, the ODE for $X_n$ implies that $\p_s X_n$ is uniformly bounded. Also, the transport equation
 \begin{equation}\label{pdeforXn}
 \p_t \, X_n +b_n(t,x)\cdot\nabla_x X_n =0
 \end{equation}
implies also that $\p_t X_n$ is bounded. Thus up to modifying $X_n$ on a Lebesgue negligible set, $X_n$ is uniformly bounded in $\textrm{Lip} ([0,T]^2\times\IR^N)$.
By Arzel\`a-Ascoli's theorem there exists $Y(s,t,x)\in \textrm{Lip} ([0,T]^2\times\IR^N)$ such that up to a subsequence $X_n(s,t,x)\to Y(s,t,x)$ locally uniformly in $[0,T]^2\times\IR^N$.
Using the identity $b_n(t,x)\cdot\nabla_x X_n=\dive (X_n\otimes b_n)-X_n\,\dive b_n$, it follows from the uniform convergence of $X_n$ and the weak convergence of $b_n$ and $\dive b_n$
that we can pass to the limit in \eqref{pdeforXn}, so that by the uniqueness of solutions to the transport equation we must have $Y=X$.
\end{proof}

Lemmas \ref{tauh} and \ref{remarkcdel}, together with Theorem \ref{fundest},  yield the following stability result for Lagrangian flows,
which states that weak convergence of the velocity fields implies that the associated flows converge strongly anyway.

\begin{propos}\label{weakstability} Let $(v_n)$ be a sequence of divergence free velocity fields uniformly bounded in $L^\infty((0,T);M^2(\IR^2))$.
Assume  that $v_n\rightharpoonup v$ in $\ID'((0,T)\times\IR^2)$, where $v\in L^\infty((0,T);M^2(\IR^2))$ is divergence free.
Assume additionally that $\curl v_n$ is uniformly equi-integrable in $L^1((0,T)\times \IR^2)$.
Let $X_n$ be the regular Lagrangian flows associated to $v_n$, and $X$  the one associated to $v$.
Then $X_n$ converges locally in measure to $X$, uniformly in $s$ and $t$.
\end{propos}

\begin{proof}
The assumptions imply that $\w_n\equiv\curl v_n$, $\w\equiv\curl v\in L^1((0,T)\times \IR^2)$,
$v_n=K*\w_n$, $v=K*\w$, thus the conditions {\bf(R1)}-{\bf(R4)} are satisfied
for $v_n$ and $v$, justifying the existence and uniqueness of $X_n$ and $X$.
We regularize $v_n$ and $v$ with respect to the spatial variable.
Take $\rho \in C_c^\infty(\IR^2)$ be the standard mollifier with spt$(\rho)\subset B_1$.
Denote by $\rho_\var(x)=\var^{-2}\rho(x/\var)$, and define
$$ 
v^\var_{n} =  v_{n}\mathop{*}_x \rho_\var, \qquad v^\var =  v\mathop{*}_x \rho_\var.
$$
Let $X_n^\var$ and $X^\var$ denote the DiPerna-Lions flows associated to $v^\var_{n}$ and $v^\var$ respectively,
as in Lemma \ref{remarkcdel}. Since $v^\var_{n}$ and $v^\var$ also satisfy {\bf(R1)}-{\bf(R4)},
it is easy to see that $X_n^\var$ and $X^\var$ are also the regular Lagrangian flows in the sense
of Definition \ref{deffloweuler}.
Then we write
\begin{equation}
	\begin{aligned}
 X_n -X & = (X_n-X_{n}^\var)+(X_{n}^\var-X^\var)+(X^\var-X) \\
& \equiv I + II + III. 
\end{aligned}
	\label{eq:decompX}
\end{equation}
By Theorem \ref{fundest} the term $III$ tends to zero locally in measure, uniformly in $s,t$, as $\var\to 0$.
For $I$, applying also Theorem \ref{fundest}, which is possible because the $\w_n$ are uniformly equi-integrable,
gives that for all $\gamma>0$, $r>0$, $\eta>0$, there exist $\lambda>0$ and $C>0$ such that
\begin{equation}\label{vne}
\La^2(B_r\cap \{|X_{n}^\var (s,t,\cdot)-X_n(s,t,\cdot)|>\gamma\})\leq C\|v^\var_{n}-v_n\|_{L^1((0,T)\times B_\lambda)}+\eta,
\end{equation}
for all $s,t\in [0,T]^2$.
Using Minkowski's inequality and applying Lemma~\ref{tauh}, we estimate
$$
\begin{aligned}
 \|v_{n}^\var- v_n\|_{L^1((0,T);L^p(\IR^2))}
 & =\int_0^T \left[\int\limits_{\IR^2}\left|\int\limits_{B_\var}[v_n(t,x-y)-v_n(t,x)\rho_\var(y)dy\right|^p dx\right]^{1/p} dt \\
&\leq \int_0^T  \int\limits_{B_\var}    \left[\int\limits_{\IR^2} |v_n(t,x-y)-v_n(t,x)|^pdx \right]^{1/p} |\rho_\var(y)|dy \, dt\\
&\leq c_p\,\|\w_n \|_{L^1((0,T);L^1(\IR^2))} \int\limits_{B_\var} |y|^\alpha|\rho_\var(y)| dy\\
&\leq C\,\var^\alpha.
 \end{aligned}
$$
Thus the first term in the right-hand side of \eqref{vne} tends to zero  as $\var\to 0$, uniformly in $n$. 
We deduce that the terms $I$ and $III$ can be made arbitrarily small independently of $n$, for a suitable choice of $\var$.
Once such $\var$ is chosen, we observe that we can apply Lemma \ref{remarkcdel}
to the vector fields $v_{n}^\var$ and $v^\var$.
We deduce that $X_n^\var\to X^\var$ locally uniformly in $s,t,x$, as $n\to\infty$,
which concludes the proof of the Proposition.
\end{proof}

\section{Existence and stability of Lagrangian solutions to the Euler system}

We now apply the stability results for Lagrangian flows derived in the previous section
in order to get stability and existence of Lagrangian solutions to the Euler equations.

\begin{theorem}[Stability of Lagrangian solutions]\label{compactness}
Let $(\w_n,v_n)  \in  C([0,T];L^1(\IR^2))\times \break C([0,T];M^2(\IR^2))$ be a sequence of Lagrangian solutions to the Euler equations (Definition \ref{lagrangiansol})
associated to uniformly in $n$ equi-integrable initial vorticities $\w_n^0$.
Let $X_n$ denote the regular Lagrangian flows associated to $v_n$.
Then, up to the extraction of a subsequence, there exists $(\w,v) \in  C([0,T];L^1(\IR^2))\times C([0,T];M^2(\IR^2))$
such that $v$ is associated to $\w$ by the convolution formula \eqref{biotsavart} and
\begin{itemize}
\item[(1)] $X_n\rightarrow X$ locally in measure, uniformly in $s,t$, where $X$ is the regular Lagrangian flow associated to $v$.
\end{itemize}
In addition,
\begin{itemize}
\item[(2)] If $\w_n^0\rightharpoonup \w^0$ weakly in $L^1(\IR^2)$, then $\w_n\rightharpoonup \w$ in $  C([0,T];L^1(\IR^2)-{\mathrm w})$,
\item[(3)] If $ \w^0_n\to \w^0 $ strongly in $L^1(\IR^2)$, then $\w_n\to\w$ in $  C([0,T];L^1(\IR^2)-{\mathrm s})$,
\item[(4)] $v_n\to v$ strongly in $C([0,T];L^1_{\loc}(\IR^2)) $.
\end{itemize}
Moreover, $(\w,v)$ is a Lagrangian solution to the Euler system with initial vorticity $\w^0$.
\end{theorem}

\begin{proof}
Since $v_n$ has zero divergence, the flow $X_n$ is measure preserving.
The formula \eqref{pushforwardomega} then implies that
$\|\w_n(t,\cdot)\|_{L^1(\IR^2)}=\|\w_n^0\|_{L^1(\IR^2)}$.
Thus $\w_n$ is uniformly bounded in $C([0,T];L^1(\IR^2))$, and $v_n$ is uniformly
bounded in $C([0,T];M^2(\IR^2))$.
Then, still the formula \eqref{pushforwardomega} implies that
$\w_n(t,\cdot)$ is uniformly in $n,t$ equi-integrable
(the smallness at infinity follows from the estimate of Remark 5.6 in \cite{jhde}).
In particular, $\w_n$ is uniformly equi-integrable in $L^1((0,T)\times\IR^2)$.

Then, according to the bound on $v_n$, up to a subsequence one has $v_n\rightharpoonup ^* v$
in $L^\infty((0,T);L^p_{\loc}(\IR^2))-{\mathrm w}*$ for $1<p<2$, with $v\in L^\infty((0,T);M^2(\IR^2))$.
Applying Proposition \ref{weakstability} gives point (1).
Extracting a new subsequence if necessary, one has $\w_n^0\rightharpoonup \w^0$ weakly in $L^1(\IR^2)$,
for some $\w^0\in L^1(\IR^2)$. One can define then the Lagrangian solution $\w\in C([0,T];L^1(\IR^2))$ to the transport
equation with initial data $\w^0$ by $\w(t,x)=\w^0(X(0,t,x))$.
Since $\w_n(t,x)=\w_n^0(X_n(0,t,x))$, the convergence of $X_n$ yields point (2) by the
same arguments as in Proposition 7.7 in \cite{jhde}. Point (3) works also with the
arguments of Proposition 7.3 in \cite{jhde}.
We deduce then that $v_n=K*\w_n\to K*\w$ in $C([0,T];L^1_{loc}(\IR^2))$, because
of the compact operator property stated in Lemma \ref{tauh}.
We deduce that $v=K*\w\in C([0,T];M^2(\IR^2))$ and point (4).
The definition of $\w$ concludes that $(\w,v)$ is a Lagrangian solution to the Euler system
with initial vorticity $\w^0$.
\end{proof}

\begin{corol}[Existence]\label{existencelagrangian} 
Let $(\w^0,v^0)\in L^1(\IR^2)\times M^2(\IR^2)$ with $\dive\,v^0=0$ and $\w^0=\curl v^0$.
Then there exists a Lagrangian solution $(\w,v)\in C([0,T];L^1(\IR^2))\times C([0,T];M^2(\IR^2))$
to the Euler system with initial data $(\w^0,v^0)$. 
\end{corol}

\begin{proof}
Let $\rho(x)\in C_c^\infty(\IR^2)$ be a standard mollifier,
and consider $\w_n^0=\rho_n*\w^0$, $v_n^0=K*\w^0_n$.
Then $\w^0_n\rightarrow \w^0$ in $L^1(\IR^2)$ and for each $n$ there exists a classical smooth solution $(\w_n,v_n)$ to
the Euler system with initial data $(\w_n^0,v_n^0)$.
For each $n$ it is also a Lagrangian solution.
Thus applying Theorem \ref{compactness}, we obtain at the limit a Lagrangian solution $(\w,v)$
with initial data $(\w^0,v^0)$. 
\end{proof}

\section{Lagrangian renormalized symmetrized solutions}

A byproduct of Theorem \ref{compactness} is the strong compactness of smooth solutions
to the Euler system, provided that the initial vorticities are uniformly equi-integrable.
The limit of such smooth sequences of solutions give rise to Lagrangian, renormalized (because all Lagrangian solutions are renormalized), symmetrized solutions.
However, we do not know if any Lagrangian solution $(\w,v)$ is necessary a symmetrized solution.
We therefore define solutions which are Lagrangian as well as symmetrized solutions.
They include in particular smooth solutions.

\begin{defin}[Lagrangian symmetrized solutions]\label{symmlagrangiansol}
Let $(\w^0,v^0)\in  L^1(\IR^2)\times M^2(\IR^2)$ with $\w^0=\curl v^0$.
We say the couple $(\w,v)$ is a Lagrangian symmetrized solution to the Euler system in $[0,T]$
with initial data  $(\w^0,v^0)$, if it is a Lagrangian solution in the sense of Definition \ref{lagrangiansol},
and $(\w,v)$ satisfies the formula \eqref{weakformh} where $\bar{H}_\phi$ is given by \eqref{barh}.
\end{defin}
According to Proposition \ref{prop equiv-sym-weak},
these solutions satisfy also the symmetrized vorticity formulation \eqref{weakform2}.
We have the following result.
\begin{propos} \label{compactnessofapprox} 
Let $(\w_n,v_n)$ be a sequence of Lagrangian symmetrized solutions to the Euler system 
and assume that $\w_n$ have uniformly in $n$ equi-integrable initial data $\w_n^0$.
Then up to a subsequence, $v_n(t,x)\to v(t,x)$ strongly in $C([0,T];L^1_{\loc}(\IR^2))$,
with $(\w,v)$ a Lagrangian symmetrized solution, and 
\begin{enumerate}
 \item If $\w_n^0\rightharpoonup \w^0$ weakly in $L^1(\IR^2)$, then $\w_n\rightharpoonup \w$ in $  C([0,T];L^1(\IR^2)-{\mathrm w})$.
 \item If $ \w^0_n\to \w^0 $ strongly in $L^1(\IR^2)$, then $\w_n\to\w$ in $  C([0,T];L^1(\IR^2)-{\mathrm s})$.
\end{enumerate}
\end{propos}

\begin{proof}
The convergence follows from Theorem \ref{compactness}.
The only new thing is that $(\w,v)$ is also symmetrized.
This follows by passing to the limit in the equation \eqref{weakformh} for $(\w_n,v_n)$.
The linear terms clearly converge, and the convergence of the nonlinear term follows from the boundedness of $\bar{H}_\phi$
and the convergence of $\w_n$ in $C([0,T];L^1(\IR^2)-{\mathrm w})$, that implies
the convergence of $w_n(t,x)w_n(t,y)$ in $C([0,T];L^1(\IR^2\times\IR^2)-{\mathrm w})$.
\end{proof}
Note that we are not able to pass to the limit in the renormalized formulation of Definition \ref{renorm},
unless strong convergence in $L^1(\IR^2)$ of $\w_n^0$ to $\w^0$ is assumed.\\

We finally conclude the existence of solutions to the Euler system in all the five senses defined in Section \ref{weaksoleulerl1}. 

\begin{propos}[Existence of Lagrangian symmetrized, and weak velocity solutions]
 Let $(\w^0,v^0)\in L^1(\IR^2)\times M^2(\IR^2)$ with $\w^0=\curl v^0$ and $\dive v^0=0$.
Then there exists a Lagrangian symmetrized solution $(\w,v)$ to the Euler system with initial data $(\w^0,v^0)$.
It is in particular a renormalized and vorticity symmetrized solution.

Under the additional assumption $v^0\in L^2_{\loc}(\IR^2)$, one can find $(\w,v)$
with the property $v\in L^\infty((0,T);L^2_{\loc}(\IR^2))$, and it is then a solution
to the weak velocity formulation \eqref{velocityform}.
\end{propos}
\begin{proof}
This follows from Proposition \ref{compactnessofapprox} after mollifying $(\w^0,v^0)$ as in Corollary \ref{existencelagrangian}.
When $v_0\in L^2_{\loc}$, the sequence of approximations is bounded in $L^\infty((0,T);L^2_{\loc}(\IR^2))$,
which yields at the limit $v\in L^\infty((0,T);L^2_{\loc}(\IR^2))$. One can then invoke
Proposition \ref{prop equiv-sym-weak} to conclude.
Another way to do is to use the arguments of \cite{delort,indiana}
to pass to the limit in the weak velocity formulation \eqref{velocityform}
(note anyway the identity provided by Lemma \ref{lemma sym-weak}).
\end{proof}
\begin{rmk}\label{rem time-compact}
In the context of Lagrangian symmetrized solutions, instead of using the general argument
of Proposition \ref{weakstability} to get stability of the flow,
it is possible to prove directly the compactness in time and space of the velocity,
by using $v=K*\w$, with $\w$ bounded in $L^\infty((0,T);L^1(\IR^2))$,
and the symmetrized vorticity formulation \eqref{weakform2} that implies
that $\partial_t\w\in L^\infty((0,T);\ID')$. These properties imply by Aubin's lemma
that $v$ is compact in $L^1_{\loc}((0,T)\times\IR^2)$.
\end{rmk}

\vspace{.3 cm}
\subsection*{Acknowledgments} This research has been partially supported by the SNSF grants 140232 and 156112.

\end{document}